\documentclass[a4paper, twoside, 12pt]{amsart}
\usepackage{amsmath}
\usepackage{amssymb}
\usepackage{graphicx} 
\usepackage{tikz-cd}
\usepackage{hyperref}
\usepackage{cleveref}

\usepackage[foot]{amsaddr}

\newtheorem{lem}{Lemma}[section]
\newtheorem{prop}[lem]{Proposition}
\newtheorem{cor}[lem]{Corollary}
\newtheorem{thm}[lem]{Theorem}

\theoremstyle{definition}
\newtheorem{defin}[lem]{Definition}
\newtheorem{remark}[lem]{Remark}

\newtheorem{example}[lem]{Example}
\newtheorem{question}[lem]{Question}
\newtheorem*{motivating question}{Motivating question}

\DeclareMathOperator{\rad}{rad}
\DeclareMathOperator{\grad}{grad}
\DeclareMathOperator{\gr}{gr}
\DeclareMathOperator{\topp}{top}

\DeclareMathOperator{\Hom}{Hom}
\DeclareMathOperator{\Ext}{Ext}
\DeclareMathOperator{\Spec}{Spec}

\renewcommand{\grad}{\operatorname{grad}}

\begin{document}

\begin{abstract}
We investigate compatibility of gradings for an almost Koszul or Koszul algebra $R$ that is also the higher preprojective algebra $\Pi_{n+1}(A)$ of an $n$-hereditary algebra $A$. 
For an $n$-representation finite algebra $A$, we show that $A$ must be Koszul if $\Pi_{n+1}(A)$ can be endowed with an almost Koszul grading.
For an acyclic basic $n$-representation infinite algebra $A$, we show that $A$ must be Koszul if $\Pi_{n+1}(A)$ can be endowed with a Koszul grading.
From this we deduce that a higher preprojective grading of an (almost) Koszul algebra $R = \Pi_{n+1}(A)$ is, in both cases, isomorphic to a cut of the (almost) Koszul grading.
Up to a further assumption on the tops of the degree $0$ subalgebras for the different gradings, we also show a similar result without the basic assumption in the $n$-representation infinite case.
As an application, we show that $n$-APR tilting preserves the property of being Koszul for $n$-representation infinite algebras. 
\end{abstract}

\title[\resizebox{4.5in}{!}{On compatibility of Koszul- and higher preprojective gradings}]{On compatibility of Koszul- and higher preprojective gradings}
\date{\today}

\author[Darius Dramburg]{Darius Dramburg}
\address{\emph{Corresponding author:} Darius Dramburg, Department of Mathematics, Uppsala University, Box 480, 751 06 Uppsala, Sweden}
\email{darius.dramburg@math.uu.se}

\author[Mads Hustad Sandøy]{Mads Hustad Sandøy}
\address{Mads Hustad Sandøy, Department of Mathematical Sciences, NTNU, NO-7491 Trondheim, Norway}
\email{mads.sandoy@ntnu.no}

\maketitle

\section{Introduction}
Preprojective algebras of quivers play an important role in representation theory. They were initially constructed by Gelfand and Ponomarev \cite{GP79}, and given a combinatorial description by Dlab and Ringel \cite{DR80}. For a quiver $Q$, the preprojective algebra is the path algebra $k\overline{Q}$ of the doubled quiver $\overline{Q}$, with relations generated by commutators of arrows $\alpha \in Q_1$ and their reverse arrows $\alpha^\ast$, i.e.\ we have 
\[ \Pi(Q) = \Pi(kQ) = k\overline{Q}/\left( \sum_{\alpha \in Q_1} \alpha \alpha^\ast - \alpha^\ast \alpha \right). \]
A purely homological description was given by Baer, Geigle and Lenzing \cite{BGL87}, which has the feature that the preprojective algebra becomes a graded algebra. In detail, for a quiver $Q$, one may construct the preprojective algebra as the tensor algebra of the $kQ$-bimodule of extensions of injectives by projectives, i.e.\ 
\[ \Pi(Q) = \Pi(kQ) = \operatorname{T}_{kQ} \Ext^1_{kQ}(D(kQ), kQ).  \]
In this way, $\Pi(Q)$ inherits a grading from tensor degrees, and one recovers $kQ$ as the degree $0$ part of this grading. Furthermore, a change of orientation of $Q$ amounts to a change of the grading on $\Pi(Q)$, so while $\Pi(Q)$ as an ungraded algebra does not retain this information, the additional structure does. 
It is this perspective on the preprojective algebra that was generalised by Iyama and Oppermann \cite{IO13} in Iyama's higher Auslander-Reiten theory. Here, for a finite dimensional algebra $A$ of global dimension $n$, the $(n+1)$-preprojective algebra is defined as 
\[ \Pi_{n+1}(A) = \operatorname{T}_A \Ext^n_A(D(A), A). \]
In the same way as for the classical case $n=1$, the $(n+1)$-preprojective algebra inherits a grading from tensor degrees, and both cases share many similarities. Dlab-Ringel's combinatorial construction of the preprojective algebra was generalised by Thibault \cite{Thi20}, and by Grant and Iyama \cite{GI19} to the setting of higher Auslander-Reiten theory. However, this combinatorial description of the higher preprojective algebra hinges on the assumption that $A$ is Koszul. Using this information, for a basic $n$-hereditary Koszul algebra $A=kQ/I$ one forms a new quiver $\overline{Q}$ by adding new arrows to $Q$ corresponding to a basis of $K_n$, the last term in the Koszul resolution of $A$. Similarly, the relations for $\overline{Q}$ are obtained from the Koszul resolution of $A$. This construction also keeps track of the higher preprojective grading on $\Pi_{n+1}(A)$ since the grading is completely determined by the newly added arrows, which have preprojective degree $1$. The analog of changing the orientation of $Q$ via reflection functors when $n=1$ is called $n$-Auslander-Platzeck-Reiten tilting ($n$-APR tilting for short), and as before this amounts to a change of preprojective grading on $\overline{Q}$. 
The interplay of Koszul and preprojective structures go further. If the Koszul algebra $A$ is $n$-representation infinite, then $\Pi_{n+1}(A)$ can be endowed with a Koszul grading, in addition to its preprojective grading. Furthermore, in the basic case, this Koszul grading is simply given by path-length in $k\overline{Q}$, and the preprojective grading by a grading on $\overline{Q}$, so the two gradings form a $\mathbb{Z}^2$-grading. If $A$ is instead $n$-representation finite and Koszul, the algebra $\Pi_{n+1}(A)$ instead becomes \emph{almost} Koszul, and similarly the two gradings form a $\mathbb{Z}^2$-grading. 

It is this interplay between the two gradings that we want to investigate in this article, both for the $n$-representation finite and the $n$-representation infinite case. The motivating question, which is made precise in \Cref{Ques: Compatibility?}, asks whether a given higher preprojective grading and a Koszul grading on some algebra $R$ need to be compatible in the way described above. We begin by explaining that this is not literally true and give some easy examples. However, the examples suggest that for two ``incompatible'' gradings we can find an automorphism that moves the higher preprojective grading to one that is compatible with the given Koszul grading. We give a proof of this statement in the $n$-representation finite case in \Cref{Thm: Main thm for n-RF algebras}. In the $n$-representation infinite case, we prove an analogous statement for basic acyclic algebras in \Cref{Thm: Main theorem for basic nRI}, and a non-basic version for acyclic algebras over a fixed semisimple base ring in \Cref{Thm: Main theorem for nonbasic nRI}. The proof is based on the fact that if $A$ is ring-indecomposable and $n$-hereditary, and $\Pi_{n+1}(A)$ can be endowed with some (almost) Koszul grading, then $A$ can be endowed with a Koszul grading. In this way, our results can be seen as a converse to Grant-Iyama's work, showing that $A$ can be given a Koszul grading if and only if $\Pi_{n+1}(A)$ can be given an (almost) Koszul grading. Additionally, this answers a question raised in \cite[Question 3.13]{DramburgGasanovaI} in the basic, $n$-representation tame case, which also applies to \cite{DramburgGasanovaII}. In the $n$-representation tame case, our results have a clear geometric interpretation. The compatibility of gradings becomes equivalent to the commutation of two $k^\ast$-actions on the affine variety $\Spec(Z(\Pi_{n+1}(A)))$, as outlined in \Cref{Rem: Geometric interpretation}.
Our results also complement the work in \cite{Thi20}, where Thibault shows that certain Koszul algebras can not arise as higher preprojective algebras. Similarly, this shows that in order to construct preprojective gradings on Koszul quiver algebras, one can exhaust all gradings up to isomorphism by grading the quiver, as has been done in \cite{HIO14, DramburgGasanovaI, DramburgGasanovaII, Giovannini}.  
As an application, we show that $n$-APR tilting preserves Koszulity. 

\section{Setup}
We work over an algebraically closed field $k$ of characteristic zero, and we denote by $D({-}) = \Hom_k({-}, k)$ the usual $k$-duality. 

\subsection{The Wedderburn-Malcev theorem}
We will use the Wedderburn-Malcev theorem in the following way. We denote by $\rad(A)$ the Jacobson radical of the algebra $A$. 

\begin{thm}\cite{malcev1942representation} \label{Thm: Wedderburn-Malcev}
    Let $A$ be a finite dimensional algebra such that $A/\rad(A)$ is separable. Then there exists a semisimple subalgebra $S \subseteq A$ such that $A = S \oplus \rad(A)$. Furthermore, for any semisimple subalgebra $T \subseteq A$, there exists $r \in \rad(A)$ such that $(1+r)^{-1} T (1+r) \subseteq S$. 
\end{thm}

A decomposition $A = S \oplus \rad(A)$ as in the theorem is called a \emph{Wedderburn-Malcev decomposition}.

\subsection{Graded algebras}
Let $R$ be a $k$-algebra. We denote by $R_\bullet = \bigoplus_{i \geq 0} R_i$ a \emph{nonnegative} $\mathbb{Z}$-grading. We write $R_+ = \bigoplus_{ i > 0} R_i$ for the positive degree part of $R_\bullet$, and more generally $R_{\geq n} = \bigoplus_{i \geq n} R_i$ for the part of degree at least $n$. The following definition and properties can be found in \cite{Nas-VanOystaeyen}.

\begin{defin}
    The \emph{graded radical} of $R_\bullet$ is the intersection of all annihilators of graded simple $R_\bullet$-modules, and denoted by 
    \[ \grad(R_\bullet) = \bigcap_{ S \text{ graded simple} } \operatorname{Ann}_{R}(S)   \]
\end{defin}

\begin{prop}\cite[Corollary 2.9.3, Proposition 5.2.4, Proposition 9.6.4]{Nas-VanOystaeyen} \label{Pro: Properties of gradings}
    Let $R_\bullet$ be nonnegatively graded.
    \begin{enumerate}
        \item We have $\grad(R_\bullet) \cap R_0 = \rad(R_0)$. 
        \item We have $\grad(R_\bullet) = \rad(R_0) \oplus R_+$. 
        \item The graded simple modules up to shift and graded isomorphism coincide with the simple $R_\bullet/\grad(R_\bullet) \simeq R_0/\rad(R_0)$ modules. 
        \item If $R_\bullet$ is finite dimensional, then $\grad(R_\bullet) = \rad(R)$. 
    \end{enumerate}
\end{prop}

We furthermore use the following notation throughout, which will be useful when dealing with several gradings on the same ring $R$. 

\begin{defin}
    For a graded ring $R_\bullet$ and graded modules $M$, $N$, we denote by 
    \[ \Hom_{R_\bullet} (M, N) \text{ and } \Ext^i_{R_\bullet} (M, N) \]
    the graded Hom- and Ext-spaces. That means, $\Hom_{R_\bullet}(M, N)$ consists of homogeneous morphisms of degree $0$. 
\end{defin}

For a graded module $M$, we define the \textit{$j$-th graded shift of $M$} to be the graded module $M\langle j \rangle$ satisfying that 
\[ (M\langle i \rangle)_j = M_{j-i},\]
and having the same module structure as $M$ as an ungraded module. 

\subsubsection{Koszul algebras}
We collect some basic facts on Koszulity that will be of relevance for us. 

\begin{defin}
    A nonnegatively graded algebra $R_\bullet = \bigoplus_{i \geq 0} R_i$ with $S= R_0$ semisimple is called \emph{Koszul} if $S$, viewed as a graded $R_\bullet$-module, has a graded projective resolution so that the $i$-th term in the resolution is generated in degree $i$. 
\end{defin}

We will only work with \emph{locally finite dimensional} gradings, so we make this blanket assumption for the rest of the article. The following characterisation of Koszulity will be useful for us, and the equivalence to the definition used in this article can be found in \cite{martinez2007introduction}. 

\begin{prop}\cite[Theorem 2.8]{martinez2007introduction}\label{Prop: Graded algebra Koszul iff Yoneda is generated in 0 and 1}
    Let $R = \bigoplus_{i \geq 0} R_i$ be a locally finite dimensional, nonnegatively graded algebra with $R_0$ semisimple. Then $R$ is Koszul if and only if $\Ext^*_R(R_0, R_0) = \bigoplus_{i \geq 0} \Ext^i_R(R_0, R_0)$ is generated in cohomological degrees $0$ and $1$. 
\end{prop}

Note that if $A$ is a finite dimensional algebra, the filtration of $A$ by powers of the Jacobson radical $\rad(A)$ gives rise to the \emph{associated graded} algebra 
\[ \bigoplus_{i \geq 0} \rad(A)^i/\rad(A)^{i+1},  \]
where multiplication of homogeneous elements is defined via 
\[ (a + \rad(A)^{i+1} ) \cdot (b + \rad(A)^{j+1}) = (ab + \rad(A)^{i+j+1}), \] where $a \in \rad(A)^{i}$ and $b \in \rad(A)^{j}$. 

\begin{cor}\label{Cor: FD algebera Koszul iff Yoneda is gen in 0 and 1}
    Let $A$ be a finite dimensional algebra, and let $S = A/\rad(A)$. Then $A$ admits a Koszul grading if and only if $\Ext^*_A(S, S)$ is generated in cohomological degrees $0$ and $1$.   
\end{cor}

\begin{proof}
    By \cite[Corollary 2.5.2]{BGS96}, any Koszul grading on a finite dimensional algebra $A$ is graded isomorphic to the grading induced by $A = \bigoplus_{i \geq 0} \rad^i(A)/\rad^{i+1}(A)$, so the previous proposition can be applied to this grading. 
\end{proof}

We also need that locally finite dimensional Koszul gradings are essentially unique. For Koszul gradings of finite dimensional algebras, this was already noted in \cite[Corollary 2.5.2]{BGS96}. We rely on the work of Gaddis, which is the non-connected version of a lemma of Bell and Zhang \cite{BellZhang}.

\begin{thm}\cite{gaddis2021isomorphisms}
    Let $Q$ and $Q'$ be finite quivers, and equip $kQ$ and $kQ'$ with the path-length grading. Let $R = kQ/I$ and $R' = kQ'/I'$, where $I$ and $I'$ are homogeneous ideals generated in degree at least $2$. If $R \simeq R'$ as ungraded algebras, then $R \simeq R'$ as graded algebras with respect to the induced path-length gradings. 
\end{thm}

\begin{cor}\label{Cor: Koszul grading is unique}
    A locally finite dimensional Koszul grading on $R$ such that $R_0$ is basic is unique up to graded isomorphism. 
\end{cor}

\begin{proof}
    Let $R_\bullet = \bigoplus_{i \geq 0} R_i$ be a locally finite dimensional Koszul grading. Then $R$ is in particular quadratic, so we have $R \simeq \operatorname{T}_{R_0} R_1/\langle V \rangle$ where $V \subseteq R_1 \otimes_{R_0} R_1$. Thus, we may write $R \simeq kQ/I$ for $Q_0$ being a choice of idempotents in $R_0$, and $Q_1$ being a choice of $R_0$-linearly independent generators in $R_1$. Then $I$ is homogeneous and generated in degree $2$, namely by the paths corresponding to elements in $V$. Doing this twice for two Koszul gradings and applying Gaddis' theorem gives a graded isomorphism for the two Koszul gradings.   
\end{proof}

Hence, in the basic case, we can fix a semisimple base ring. 

\begin{example}
    Consider the polynomial ring $R = k[x] = k[x+1]$, with two gradings placing $x$ respectively $x+1$ in degree $1$. Then the degree $0$ parts are not isomorphic as simple $R$-modules. Of course, the automorphism exchanging $x$ and $x+1$ is a graded isomorphism for the two gradings, so we can adjust one grading so that the degree $0$ parts coincide. 
\end{example}

\subsubsection{Almost Koszul algebras}
Next, we give a brief summary of almost Koszul algebras, as introduced by Brenner, Butler and King \cite{BBK02}. 

\begin{defin}
    A nonnegatively graded algebra $R_\bullet = \bigoplus_{i \geq 0} R_i$ with $S = R_0$ semisimple is called \emph{almost Koszul}, or \emph{$(p,q)$-Koszul}, if there exist integers $p,q \geq 1$ such that $R_n = 0$ for all $n>p$ and such that there is an exact sequence
    \[ 0 \to S' \to P_q \to  \cdots \to P_0 \to S \to 0   \]
    so that $P_i$ is projective and generated in degree $i$ and $S'$ is semisimple and concentrated in degree $p+q$. 
\end{defin}

Since we only consider locally finite gradings, it follows immediately that an almost Koszul algebra $R_\bullet$ is finite dimensional. We furthermore have the following observation. 

\begin{remark}
    Let $R_\bullet$ be almost Koszul. Then the Jacobson radical and the graded radical coincide, that is we have $\rad(R) = \grad(R_\bullet) = R_{\geq 1}$. Therefore, $R_\bullet \simeq \bigoplus_{i \geq 0} \rad(R)^i/\rad(R)^{i+1}$ is a graded isomorphism. 
\end{remark}

\subsection{\texorpdfstring{$(n+1)$}{(n+1)}-preprojective algebras}
We are interested in how an (almost) Koszul grading interacts with a higher preprojective grading, so we summarise the construction of such a grading here. Let $A$ be $n$-hereditary, then Iyama and Oppermann defined in \cite{IO13} the $(n+1)$-preprojective algebra of $A$ as 
\[ \Pi_{n+1}(A) = \operatorname{T}_A \Ext^n_A(D(A), A) \]
where $D(A) = \Hom_k(A,k)$ is the usual $k$-duality. The higher preprojective algebra has a grading induced from tensor degree, which we refer to as the \emph{higher preprojective grading}. Note that $\Pi_{n+1}(A)$ has some special properties as a graded algebra depending on whether $A$ is $n$-representation finite or $n$-representation infinite. 

\begin{remark}
    If $A$ is $n$-representation finite, then $R= \Pi_{n+1}(A)$ is finite dimensional. We denote the higher preprojective grading by ${}_\bullet R$. Suppose there exists some other nonnegative grading $R_\bullet$ on $R$. Then it follows from \Cref{Pro: Properties of gradings} that 
    \[ \grad({}_\bullet R ) = \rad(R) = \grad(R_\bullet). \]
\end{remark}

In the $n$-representation infinite case, $\Pi_{n+1}(A)$ is infinite dimensional. The relevant properties of $\Pi_{n+1}(A)$ are summarised in the following definition. 

\begin{defin}
    A nonnegatively graded algebra ${}_\bullet R = \bigoplus_{i \geq 0} {}_i R$ is called bimodule $(n+1)$-Calabi-Yau of Gorenstein parameter $a$ if there exists a bounded graded projective bimodule resolution $P$ of ${}_\bullet R$ with an isomorphism 
    \[ P \simeq \Hom_{R^e}(P, R^e)[n+1]\langle -a \rangle   \]
    of complexes of graded bimodules, where $[n+1]$ is the cohomological shift and $\langle -a \rangle$ is degree shift. 
\end{defin}
 
The characterisation of higher preprojective algebras of $n$-representation infinite algebras is as follows. 

\begin{thm}
    There is a bijection between $n$-representation infinite algebras $A$ up to isomorphism and locally finite dimensional bimodule $(n+1)$-Calabi-Yau algebras ${}_\bullet R$ of Gorenstein parameter $1$, both sides taken up to isomorphism. The bijection is given by 
    \[ A \mapsto \Pi_{n+1}(A) \text{ and } {}_\bullet R  \mapsto {}_0 R. \]
\end{thm}

\begin{remark}\label{Rem: On graded coherence}
    An important property a higher preprojective algebra ${}_\bullet R$ may possess is that of being \emph{graded coherent}. This means that the category of finitely presented graded left ${}_\bullet R$-modules is closed under kernels. The coherence property and its implications for higher Auslander-Reiten theory have been studied in \cite{Min12, MM11} as well as \cite[Section 6]{HIO14}, and the question when the higher preprojective algebra is graded coherent was posed in \cite[Question 4.37]{HIO14}. We note that for example if $A$ is \emph{$n$-representation tame}, then $\Pi_{n+1}(A)$ is graded coherent.  
    However, note that \cite{Gelinas-thesis} has used a relationship between certain kinds of $t$-structures on stable categories of maximal Cohen--Macaulay modules and the notion of absolute Koszulity to construct an infinite family of higher preprojective algebras that are not graded coherent. 
    We therefore point out that we do not assume graded coherence of the higher preprojective algebra.
\end{remark}

The discussion of higher preprojective algebras shows that in both the $n$-representation finite and the $n$-representation infinite case, we can start out with an ungraded algebra $R$, satisfying some natural assumptions, and ask whether it can be equipped with an additional grading making it higher preprojective. This is particularly interesting in the $n$-representation infinite case, see \cite{Thi20}, and many constructions of $n$-representation infinite algebras proceed in this way. 

\subsection{The Grant-Iyama construction}
In this subsection, we recall a construction due to Grant and Iyama in \cite{GI19}. For this, we start with an $n$-hereditary algebra $A$, and furthermore assume that it has a Koszul grading $A_\bullet$. With respect to the grading on $A$, the generating bimodule $\Ext^n_A(D(A), A)$ of the preprojective algebra becomes graded via  
\[ \Ext^n_A(D(A), A) = \bigoplus_{i \in \mathbb{Z}} \Ext^n_{\gr A_\bullet}(D(A), A\langle j \rangle ).  \]
This way, the $(n+1)$-preprojective algebra 
\[ \Pi = \Pi_{n+1}(A) = \operatorname{T}_A \Ext^n_A(D(A), A), \]
becomes $\mathbb{Z}^2$-graded. The first part of the grading is the usual higher preprojective grading induced from tensor degrees, while the second part of this grading is induced from the grading on $\Ext^n_A(D(A), A)$ and naturally extended to all tensor powers $\Ext^n_A(D(A), A)^{\otimes i}$. 

This $\mathbb{Z}^2$-grading gives rise to the \emph{$(n+1)$-total grading}, given by 
\[ \Pi(A)_{m} = \bigoplus_{(n+1)i + j = m} \Pi_{i,j}. \]

\begin{thm}\cite[Theorem 4.21]{GI19} \label{Thm: GI construction}
    Let $A_\bullet$ be $n$-hereditary and Koszul. 
    \begin{enumerate}
        \item If $A$ is $n$-representation finite, then $\Pi(A)$ with the $(n+1)$-total grading is $(p,n+1)$-Koszul, where $p = \max \{ i \geq 0 \mid \Pi_i \neq 0 \} $.
        \item If $A$ is $n$-representation infinite, then $\Pi(A)$ with the $(n+1)$-total grading is Koszul.
    \end{enumerate}
\end{thm}

The authors also note the following from the construction of the grading. 

\begin{lem}\cite[Proposition 4.16]{GI19} \label{Lem: E is generated in deg -n}
    If $A_\bullet$ is Koszul, then the $A$-bimodule $\Ext^n_A(D(A), A)$ is generated in degree $-n$ with respect to the Koszul grading. 
\end{lem}

\subsection{Higher preprojective cuts}
We are now ready to make our motivating question precise. We have seen that if an $n$-hereditary algebra $A$ is Koszul, then $\Pi_{n+1}(A)$ can be endowed with an additional grading, the $(n+1)$-total grading. Furthermore, it is easy to see that these two gradings form a $\mathbb{Z}^2$-grading. Conversely, given an (almost) Koszul graded algebra $R_\bullet$, one may try to give it a higher preprojective grading. In the basic case, it is tempting to use the (almost) Koszul grading to build a quiver presentation for $R$, and then grade the quiver directly. We give this kind of grading a special name. 

\begin{defin}
    A higher preprojective grading ${}_\bullet R$ on an (almost) Koszul algebra $R_\bullet$ is called a \emph{cut} if the following are satisfied: 
    \begin{enumerate}
        \item The (almost) Koszul generating bimodule $R_1$ decomposes into preprojectively homogeneous pieces of degrees $0$ and $1$, i.e.\ we have $R_1 = {}_0(R_1) \oplus {}_1(R_1)$ such that ${}_i(R_1) \leq {}_iR$ as vector spaces.  
        \item The decomposition recovers the preprojective grading, i.e.\ we have ${}_0R = \langle R_0 \oplus {}_0(R_1) \rangle $ as algebras and ${}_1R$ is generated by ${}_1(R_1) $ as an ${}_0R$-bimodule. 
    \end{enumerate}
\end{defin}

Note that in this definition, being a cut depends on the chosen (almost) Koszul grading. The first noteworthy consequence of this definition is that a cut induces a bigrading on $R$. 

\begin{prop}
    Let ${}_\bullet R$ be a cut on the (almost) Koszul algebra $R_\bullet$. Then the two gradings induce a $\mathbb{Z}^2$-grading.    
\end{prop}

\begin{proof}
   We first show that each (almost) Koszul homogeneous piece $R_i$ decomposes into a direct sum of preprojectively homogeneous pieces. By assumption, the decomposition is true for $R_0$ and $R_1$. Furthermore, it is obvious that $R_i \supseteq \left( \bigoplus_{j \geq 0} R_i \cap {}_jR \right)$. Then note that $R_i$ is generated over $R_0$ by products of length $i$ of elements in $R_1$, and use distributivity to see that $R_i \subseteq \left( \bigoplus_{j \geq 0} R_i \cap {}_jR \right)$. 
   Now, we show that a preprojectively homogeneous piece ${}_i R$ decomposes into (almost) Koszul homogeneous components. The argument is the same, we note ${}_0 R = \langle R_0 \oplus {}_0(R_1) \rangle$ is generated by (almost) Koszul homogeneous elements, and the same holds for ${}_1R = ({}_0R) ({}_1(R_1)) ({}_0R)$. Then, writing each element in ${}_j R$ as a product of length $j$ of elements in ${}_1 R$ and using distributivity shows the claim. 
\end{proof}

\begin{remark} Let us note why the terminology of a cut is justified.
    \begin{enumerate}
        \item In the case of a basic algebra $R$, this recovers the usual quiver definition as used in \cite{Iyama-Oppermann, HIO14}: We can identify the vertices $Q_0$ with idempotents in $R_0$, the arrows in $Q_1$ with appropriate generators in $R_1$, and obtain a graded isomorphism $R_\bullet \simeq kQ/I$, where $kQ/I$ is graded with respect to path-length. A cut then amounts to grading the arrows of $Q$. In detail, a cut is a subset $C \subseteq Q_1$, and removing it gives a cut quiver $Q_C$. Then, we have that ${}_0R \simeq kQ_C/I_C$, where $I_C$ is the set of relations where all relations involving arrows in $C$ have been removed. Furthermore ${}_1 R$ is generated over $kQ_C/I_C$ by $kC$, up to the specified isomorphism. 
        \item When starting instead with a Koszul grading on $A$, which is either assumed to be higher representation finite or higher representation infinite, the $(n+1)$-total grading on $\Pi_{n+1}(A)$ constructed by Grant and Iyama is an (almost) Koszul grading such that the higher preprojective grading is a cut.
\end{enumerate}
\end{remark}

We make the last remark about the $(n+1)$-total grading from above precise.

\begin{prop}\label{Prop: GI grading defines a cut}
    Let $A$ be Koszul and $n$-hereditary, and $\Pi_{n+1}(A) = R_\bullet = \bigoplus_{i \geq 0} R_i$ the higher preprojective algebra with the $(n+1)$-total (almost) Koszul grading constructed by Grant-Iyama. Then the preprojective grading on $R_\bullet$ is a cut with respect to this grading. 
\end{prop}

\begin{proof}
    Denote by $E = \operatorname{Ext}^n_A(D(A), A)$ the bimodule which generates the preprojective algebra. Then it follows from \Cref{Lem: E is generated in deg -n} that $E$ as a graded $A$-module is generated in degree $-n$. One sees easily that the Koszul degree $0$ piece is $R_0 = A_0$, and that $R_1 = A_1 \oplus E_{-n}$. Since $A$ is Koszul, it follows that $A$ is generated in degree $1$, and $E$ is generated over $A$ by $R_1 \cap E = E_{-n}$. 
\end{proof}

\begin{example}\label{Ex: Kronecker example}
    Not every higher preprojective grading is a cut. To see this, simply take a higher preprojective cut ${}_\bullet R$ on the Koszul graded algebra $R_\bullet$ and an arbitrary automorphism $\varphi \in \operatorname{Aut}(R)$. Then we obtain a new higher preprojective grading on $R$ from the higher preprojective grading on $R = \Pi_{n+1}(\varphi(A))$ and this new grading in general is not a cut. More explicitly, consider the hereditary, representation infinite algebra $A = kQ$ where $Q$ is the Kronecker quiver 
    \[\begin{tikzcd}[ampersand replacement=\&]
	e \& f
	\arrow["b"', shift right, from=1-1, to=1-2]
	\arrow["a", shift left=2, from=1-1, to=1-2]
\end{tikzcd}. \]
    Here, we already labeled the idempotents by $e$ and $f$ respectively. The (classical) preprojective algebra $R = \Pi_2(A)$ is given by adding the reverse arrows $a^\ast$ and $b^\ast$ and introducing commutativity relations.
    \[\begin{tikzcd}[ampersand replacement=\&]
	e \& f
	\arrow["b"', shift right=5, bend right =35, from=1-1, to=1-2]
	\arrow["a", shift left=5, bend left= 35, from=1-1, to=1-2]
	\arrow["{b^\ast}"', shift left=3,  bend left= 35, from=1-2, to=1-1]
	\arrow["{a^\ast }", shift right=3, bend right =35, from=1-2, to=1-1]
\end{tikzcd}\]
    Furthermore, it is graded by path-length and becomes Koszul with this grading. Then placing $\{a^\ast, b^\ast\}$ or $\{a, b\}$ in preprojective degree $1$ defines two higher preprojective cuts. Next, note that $a^2 = 0$, so $r = (1+a)$ is invertible, and conjugation by $r$ gives us a new quiver presentation for $\Pi(A)$ 
    \[\begin{tikzcd}[ampersand replacement=\&]
	e+a \& f-a
	\arrow["b"', shift right=5, bend right =35, from=1-1, to=1-2]
	\arrow["a", shift left=5, bend left= 35, from=1-1, to=1-2]
	\arrow["{(b^\ast)^r}"', shift left=3,  bend left= 35, from=1-2, to=1-1]
	\arrow["{(a^\ast)^r }", shift right=3, bend right =35, from=1-2, to=1-1]
\end{tikzcd}\]
    where the newly added arrows $a^\ast$ and $b^\ast$ get mapped to 
    \begin{align*}
        (a^\ast)^r = (1-a) a^\ast (1+a) &= a^\ast - a a^\ast + a^\ast a - a a^\ast a \\
        (b^\ast)^r = (1-a) b^\ast (1+a) &= b^\ast - a b^\ast + b^\ast a - a b^\ast a.
    \end{align*}
    In particular, the Koszul degree $1$ piece $R_1 = \langle a, b, a^\ast, b^\ast \rangle$ differs from $\langle a, b, (a^\ast)^r, (b^\ast)^r  \rangle$, but declaring $\{ a, b\}$ to have preprojective degree $1$ still defines a preprojective grading. Denote its degree $0$ part by $B = \langle e+a, f-a, (a^\ast)^r, (b^\ast)^r \rangle$. Then even the semisimple base rings differ, since $R_0 \cap B = 0$. However, we note that the graded radical for the Koszul grading is $\grad(R_\bullet) = ( a,b,a^\ast, b^\ast )$. This equals the graded radical of the newly defined higher preprojective grading, which is given by the ideal $(a, b,(a^\ast)^r, (b^\ast)^r )$. The equality is not a coincidence, as \Cref{Cor: Nilpotent generators imply equal gradicals} shows. 
\end{example}

It follows from \Cref{Cor: Koszul grading is unique} that many of the (almost) Koszul gradings we are interested in --- i.e. locally finite dimensional, and given by a connected quiver --- must also be unique up to graded isomorphism, so we can hope that this ambiguity is the only obstruction to being a cut. More formally, we ask the following. 

\begin{question} \label{Ques: Compatibility?} 
    Let $R_\bullet$ be an (almost) Koszul algebra, and suppose there exists a higher preprojective grading on $R$. Does there exist an automorphism of $R$ that maps the higher preprojective grading to a cut?
\end{question}

It was proven in \cite{HIO14} that a ring-indecomposable $n$-hereditary algebra is either $n$-representation finite or $n$-representation infinite. This is reflected in the dimension of the higher preprojective algebra being finite or infinite. We therefore split the above question into the two cases of finite and infinite as well. We will give a positive answer to this question for the $n$-representation finite case. We also prove this for basic algebras in the $n$-representation infinite case. The proof hinges on the fact that the preprojective degree $0$ part can be given a Koszul grading, which may differ from the (almost) Koszul grading on the preprojective algebra. 

\begin{remark}\label{Rem: Geometric interpretation}
    Let us give an alternative interpretation of \Cref{Ques: Compatibility?} in geometric terms. For simplicity, we restrict to the case when $A$ is $n$-representation tame. By \cite[Definition 6.10]{HIO14}, this means that $\Pi_{n+1}(A)$ is a finitely generated module over a central Noetherian subalgebra $Z \subseteq Z(\Pi_{n+1}(A))$. Note that this definition is independent of a Koszul or higher preprojective grading. A grading on $\Pi_{n+1}(A)$ induces a grading on the center $Z(\Pi_{n+1}(A))$, which is equivalent to a $k^\ast$-action on $X = \Spec(Z(\Pi_{n+1}(A)))$. This means that the Koszul and the higher preprojective gradings define two $k^\ast$-actions on $X$. In general, these actions need not commute. As with gradings, one may take two commuting actions and move one of them by an automorphism of $X$ to obtain non-commuting actions. However, if the gradings form a $\mathbb{Z}^2$-grading, then the actions do commute. From this perspective, \Cref{Ques: Compatibility?} asks whether the two given $k^\ast$-actions commute up to an automorphism of $X$.   
\end{remark}

\section{Reduction to the graded basic case}
It will be necessary to work with some `basicness' assumption. 
However, since we work with graded algebras this basicness needs to be taken with the adjective graded as well, and we explain this reduction here. 

\begin{defin}
    We call a locally finite, nonnegatively graded algebra $R_\bullet$ \emph{graded basic} if the finite dimensional algebra $R_0$ is basic.
\end{defin}

\begin{remark}
    Since graded simple $R_\bullet$ modules are, up to isomorphism and shift, given by simple $R_0$ modules, it follows that the graded simple modules of a graded basic algebra are one dimensional. 
\end{remark}

\begin{lem} \label{Lem: Gradings on semisimple are trivial}
    Let $S$ be a finite dimensional semisimple algebra. Then the following hold.
    \begin{enumerate}
        \item Any nontrivial ideal $I \trianglelefteq S$ contains a nontrivial idempotent.
        \item Any nonnegative grading on $S$ is trivial.
    \end{enumerate}
\end{lem}

\begin{proof} \leavevmode
    \begin{enumerate}
        \item Let $I \trianglelefteq S$ be a nontrivial ideal. By the Artin-Wedderburn theorem, $S \simeq \Pi_{i = 1}^m k^{n_i \times n_i}$ is a product of matrix rings over $k$. Every matrix ring $k^{n_i \times n_i}$ has only trivial ideals, so $I$ needs to contain an ideal isomorphic to some $k^{n_i \times n_i}$, and hence contains nontrivial idempotents.
        \item Consider a nonnegative grading $S = \bigoplus_{i \geq 0} S_i$. Since $S$ is finite dimensional, it follows that the grading has finite support, that means we can write $S = \bigoplus_{i = 0}^m S_i$ for some $m \geq 0$. In particular, it follows that $I = \bigoplus_{i = 1}^m S_i $ is an ideal in $S$, and furthermore that $I$ is nilpotent. From the previous argument, we know that if $I$ was nontrivial, then it would contain a nontrivial idempotent $e \in I$. However, by nilpotency of $I$ we have $e = e^{m+1} = 0$, which contradicts the assumption that $I \neq 0$. \qedhere
    \end{enumerate}
\end{proof}

Next, we consider the appropriate version of Morita equivalence for our setting. For a detailed discussion of different notions of graded Morita equivalence, we refer the reader to \cite{Abrams-Ruiz-Tomforde23}.  

\begin{prop}
    Let $R_\bullet$ be locally finite and nonnegatively graded. Let $e \in R_0$ be a full idempotent such that $eR_0e$ is basic. Then $R_\bullet^b = eRe = \bigoplus_{i \geq 0} eR_i e $ is a graded basic algebra and Morita equivalent to $R$. 
\end{prop}

\begin{proof}
    By assumption, $e$ is homogeneous, and hence $eRe = \bigoplus_{i \geq 0} eR_i e$ is a graded algebra. Furthermore, by assumption $eR_0e$ is basic so $R_\bullet^b$ is graded basic. Lastly, note that $R_0 e R_0 = R_0$ since we assumed that $e$ is full in $R_0$. However, since $1 \in R_0$, it follows that $ReR = R$, so $e$ is full in $R$, and hence $R$ and $eRe$ are Morita equivalent.   
\end{proof}

\begin{remark}
    Recall from classical Morita theory that $eR_0e$ is independent of the choice of $e \in R_0$. More precisely, for any full idempotent $f \in R_0$ such that $fR_0 f$ is basic, we have $e R_0 e \simeq fR_0 f$. Therefore, we simply write $R_\bullet^b$ without referring to $e$. 
\end{remark}

We show that we can recover $R_\bullet$ as a graded algebra from $R_\bullet^b$. Note that this can be seen as a graded Morita equivalence. For this we need idempotents that give rise to a particular decomposition of the regular module, which we call \emph{isotypic}: We decompose 
$$1 = e_{1,1} + e_{1,2} + \ldots + e_{1, d_1} + e_{2,1} + \ldots + e_{m, d_m} \in R_0$$ 
into pairwise orthogonal, primitive idempotents, where the indices are chosen so that $Re_{i_1,j_1} \simeq Re_{i_2, j_2}$ if and only if $i_1 = i_2$. We call $e = \sum_{i = 1}^m e_{i,1}$  the corresponding \emph{Morita idempotent}.  

\begin{prop}\label{Prop: Morita recovers grading}
    Let $R_\bullet$ be locally finite and nonnegatively graded. Let $1 = e_{1,1} + e_{1,2} + \ldots + e_{1, d_1} + e_{2,1} + \ldots + e_{m, d_m} \in R_0$ be an isotypic decomposition of $1$ in $R_0$ and $e = \sum_{i = 1}^m e_{i,1}$ the corresponding Morita idempotent. Then $\operatorname{End}_{eRe}( \bigoplus_{i = 1}^m (eRe_{i,1})^{d_i})$ is a graded algebra such that
    \[ R_\bullet \simeq  \operatorname{End}_{eRe} \left( \bigoplus_{i = 1}^m (eRe_{i,1})^{d_i} \right) \]
    as graded algebras. 
\end{prop}

\begin{proof}
    We recall from Morita theory that we obtain a Morita equivalence between $eRe$ and $R$ given by 
    \[ eRe \simeq \operatorname{End}_R(Re) \text{ and } R \simeq \operatorname{End}_{eRe}(eR).  \]
    Note that $e$ is homogeneous of degree $0$. Therefore, $eRe$ is graded, and $eR$ is a graded $eRe$-module. This way, $\operatorname{End}_{eRe}(eR)$ becomes graded and $R_\bullet \simeq \operatorname{End}_{eRe}(eR) $ is a graded isomorphism. Next, we decompose the $eRe$-module $eR$ further as $eR = \bigoplus_{i,j} eRe_{i,j}$ according to the decomposition of $1$ in $R_0$. From the isotypic decomposition, we obtain that there exist elements $E_{i, j_1, j_2} \in R_0$ such that $ e_{i,j_1} E_{i,j_1, j_2} = e_{i, j_2}$. Since $E_{i,j_1, j_2} \in R_0$, we conclude that the $eRe$-modules $eRe_{i,j_1}$ and $eRe_{i,j_2}$ are graded isomorphic. Hence, we can collect the graded modules $\bigoplus_{j = 1}^{d_i} eRe_{i, j} \simeq (eRe_{i,1})^{d_i}$ and obtain the graded isomorphism 
    \[  R_\bullet \simeq \operatorname{End}_{eRe}(eR) \simeq \operatorname{End}_{eRe}( \bigoplus_{i = 1}^m (eRe_{i,1})^{d_i}). \qedhere \]
\end{proof}

Next, we want to transfer an additional grading on the graded algebra $R_\bullet^b$ back to $R_\bullet$. For this, we need that the two gradings form a $\mathbb{Z}^2$-grading. We keep working with the same isotypic decomposition and Morita idempotent. 

\begin{prop}
    Let $R_\bullet$ be locally finite and nonnegatively graded so that $R_0$ is semisimple. Let $R_\bullet^b = R_{\bullet, \ast}^b$ be nonnegatively $\mathbb{Z}^2$-graded so that the $\bullet$-component is the grading induced from $R_\bullet$. Then there exists a $\mathbb{Z}^2$-grading on $R = R_{\bullet, \ast}$ so that the first component is the grading $R_\bullet$. 
\end{prop}

\begin{proof}
    Let $e \in R_0$ be the Morita idempotent that defines the graded basic algebra $R_\bullet^b = eRe$. Decompose $e = e_1 + \ldots + e_m$ inside of $R_0$ into primitive idempotents. It then follows from \Cref{Prop: Morita recovers grading} that we have a graded isomorphism $R \simeq \operatorname{End}_{eRe}(\bigoplus_{i = 1}^m (eRe_i)^{d_i})$ for some multiplicities $d_i$. Since $eRe$ is graded basic, we have that $eR_0e$ is basic and finite dimensional. The second component of the bigrading $(\bullet, \ast)$ on $eRe$ defines a grading on $eR_0e$. However, it follows from \Cref{Lem: Gradings on semisimple are trivial} that this grading is trivial on $eR_0e$, and hence we have $R^b_{(0,0)} = \bigoplus_{i \geq 0} R^b_{(0, i)}$. Recall that the identity in $eRe$ is $e$, which consequently has bidegree $(0, 0)$, as do all the $e_i$. It follows that each projective $R^b_\bullet$ module $eRe_i$ is bigraded. Hence, $R \simeq \operatorname{End}_{eRe}(\bigoplus_{i = 1}^m (eRe_i)^{d_i})$ becomes $\mathbb{Z}^2$-graded as well. Finally, it follows from \Cref{Prop: Morita recovers grading} that the first component of the bigrading on $\operatorname{End}_{eRe}(\bigoplus (eRe_i)^{d_i})$ recovers the $\bullet$-grading on $R$.       
\end{proof}

Lastly, we apply this to the case of a higher preprojective cut. 

\begin{prop} \label{Prop: HPC pulls back along Morita}
    Let $R_\bullet$ be a locally finite Koszul algebra. Suppose that there exists a higher preprojective cut on $R_\bullet^b$. Then $R_\bullet^b = R^b_{\bullet, \ast}$ becomes $\mathbb{Z}^2$-graded, and the induced $\mathbb{Z}^2$-grading on $R_\bullet$ is a higher preprojective cut. 
\end{prop}

\begin{proof}
    The fact that the second component of the grading $R_{\bullet, \ast}$ is a higher preprojective grading follows from \cite[Proposition 4.12]{Thi20}, whose proof applies for both the $n$-representation finite and the $n$-representation infinite case. To see that $R_\ast$ is a cut for the Koszul grading $R_\bullet$, we first note that clearly $R_1 = \bigoplus_{j \geq 0} R_{(1,j)}$ decomposes into preprojectively homogeneous pieces. Furthermore, recall that any idempotent $e \in R_0$ affording the Morita equivalence with $R_\bullet^b$ is full. Therefore, it follows that $R_{(1, j)} = 0$ for all $j \geq 2$, because the same is true for the cut on $R^b_\bullet$. To check the second property of a cut, note that $R_0$ becomes a graded algebra with respect to the $\ast$-component of the bigrading. By \Cref{Lem: Gradings on semisimple are trivial}, it follows that $R_0 = R_{(0,0)}$, so we have that $R_0 \subseteq \bigoplus_{i \geq 0} R_{(i,0)}$ has preprojective degree $0$. Since $R_\bullet$ is Koszul, it is generated in degrees $0$ and $1$, and hence we have that $\bigoplus_{i \geq 0} R_{(i, 0)} $ is generated by $R_{(0,0)} \oplus R_{(1, 0)} = R_0 \oplus R_{(1, 0)}$.  
\end{proof}

Note that while this reduction works in both the $n$-representation finite and the $n$-representation infinite case, we will only need it in the $n$-representation infinite case, where the following complication occurs. 

\begin{remark}
    Suppose $R_\bullet$ and ${}_\bullet R$ are a Koszul and a higher preprojective grading respectively, so that $A = {}_0 R$ is $n$-representation infinite. We may choose $e \in R_0$ for the reduction to the graded basic algebra $R_\bullet^b$. However, in general $e \not \in {}_0 R$ need not be preprojectively homogeneous. This means that the subspaces $e ({}_i R) e$ need not define a grading on $eRe$, so we can not transfer the higher preprojective grading this way. In order to remedy this, we set $f = {}_0 e$ and in this way obtain a full, homogeneous idempotent in ${}_0 R$, which gives rise to a graded basic algebra $fRf = \bigoplus_{i \geq 0} (f ({}_i R) f)$. However, since $R$ is infinite dimensional, the reductions $eRe$ and $fRf$ need not be isomorphic. We therefore will assume that $R_0 \subseteq {}_0 R$ in \Cref{Thm: Main theorem for nonbasic nRI}, so that the idempotent $e \in R_0$ is also homogeneous with respect to the preprojective grading. As pointed out in \Cref{Ex: Kronecker example}, this assumption is not fulfilled in general. In the spirit of the Wedderburn-Malcev theorem, one can hope that $R_0$ can be moved into ${}_0 R$ by some automorphism, so we record the following question. 
\end{remark} 

\begin{question}
    Let $R_\bullet$ be locally finite and nonnegatively graded. Let $e \in R$ be some idempotent, and decompose $e = e_0 + e_+$. Then $e_0$ is idempotent. Does there exist an automorphism of $R$ that takes $e$ to $e_0$. If $\{ e_1, \ldots, e_d\}$ is a complete set of primitive orthogonal idempotents, does there exist an automorphism taking it to $\{ (e_1)_0, \ldots, (e_d)_0\}$? 
\end{question}

\section{The \texorpdfstring{$n$}{n}-representation finite case}
In this section, we present the details of the $n$-representation finite case. The general strategy here will be the same as in the $n$-representation infinite case. However, the $n$-representation finite case has the feature that the higher preprojective algebra is finite dimensional, so some subtleties from the $n$-representation infinite case do not yet appear here. Let $R = \Pi_{n+1}(A)$ be the higher preprojective algebra of an $n$-representation finite algebra $A$. Denote the higher preprojective grading on $R$ by ${}_\bullet R$. Let $R_\bullet$ be another grading on $R$ so that $R_\bullet$ is $(p, n+1)$-Koszul. We want to show that $A$ can be endowed with a Koszul grading, despite the fact that $A$ is not necessarily a graded subalgebra of $R_\bullet$. Nonetheless, the graded radicals for the almost Koszul and the higher preprojective grading coincide 

\begin{remark}
    Since $A$ is $n$-representation finite, $R = \Pi_{n+1}(A)$ is finite dimensional. It follows from \Cref{Pro: Properties of gradings} that 
    \[ \grad( R_\bullet) = \rad(R) =  \grad({}_\bullet R).\] 
    In particular, since $\grad({}_\bullet R) = \rad(A) \oplus {}_+ R$, it follows that 
    \[R_0 \simeq R/R_+ = R/\grad(R_\bullet) = R/\grad( {}_\bullet R) \simeq A/\rad(A)\] 
    carries a natural $A$-module structure, and also a graded ${}_\bullet R$-module structure.
\end{remark}

\begin{prop}\label{Pro: n-RF algebra inherits Koszul}
    Let $A$ be $n$-representation finite, and suppose there exists a $(p, n+1)$-Koszul grading $R_\bullet$ on $R = \Pi_{n+1}(A)$. Then $A$ is Koszul. 
\end{prop}

\begin{proof}
    Recalling the assumption of $R_\bullet$ being $(p, n+1)$-Koszul, we consider the algebra 
    \[ T = \bigoplus_{ i = 0}^{n+1} \operatorname{Ext}^i_R(R_0, R_0).  \]
    By \cite[Remark 3.12]{BBK02}, this is the quadratic dual of $R_\bullet$. Since $R$ is finite dimensional, we furthermore have that 
    \[ \operatorname{Ext}^i_R(R_0, R_0) = \bigoplus_{j \in \mathbb{Z}} \operatorname{Ext}^i_{{}_\bullet R}(R_0, R_0\langle j \rangle).  \]
    Hence, $T$ becomes $\mathbb{Z}^2$-graded, and we refer to the cohomological grading and the internal grading respectively. Taking the degree $0$ part of the internal grading, we obtain 
    \[ \operatorname{Ext}^i_{{}_\bullet R}(R_0, R_0) = \operatorname{Ext}^i_{A} (R_0, R_0).  \]
    Since $R_\bullet$ is $(p, n+1)$-Koszul, by \cite[Proposition 3.2]{BBK02} the algebra $T$ is generated in cohomological degrees $0$ and $1$, while the full Ext-algebra $\bigoplus_{ i \geq 0} \operatorname{Ext}^i_R(R_0, R_0)$ is generated by its cohomological degree $0$ and $1$ pieces together with $\operatorname{Ext}_R^{n+2}(R_0, R_0)$. However, since $A$ has global dimension $n$, it follows that $\operatorname{Ext}_A^{n+2}(R_0, R_0) = 0$, and hence $\bigoplus_{ i \geq 0} \operatorname{Ext}^i_A(R_0, R_0)$ is generated in cohomological degrees $0$ and $1$. By \Cref{Cor: FD algebera Koszul iff Yoneda is gen in 0 and 1}, it follows that $A$ is a Koszul algebra.
\end{proof}

Now we can use the $(n+1)$-total grading constructed by Grant-Iyama to show that we can find a cut on $R_\bullet$. 

\begin{thm}\label{Thm: Main thm for n-RF algebras}
    Let $R_\bullet$ be $(p, n+1)$-Koszul such that there exists a higher preprojective grading ${}_\bullet R$, where $A= {}_0 R$ is $n$-representation finite. Then there exists a higher preprojective cut on $R_\bullet$ whose degree $0$ component is isomorphic to $A$. 
\end{thm}

\begin{proof}
    Using \Cref{Pro: n-RF algebra inherits Koszul}, we obtain that $A$ is Koszul. Therefore, we can compute the $(n+1)$-total grading $\Pi_{n+1}(A)_\bullet$ on the higher preprojective algebra $\Pi_{n+1}(A)$. The higher preprojective grading then is a cut with respect to the $(n+1)$-total grading by \Cref{Prop: GI grading defines a cut}. Note that $R \simeq \Pi_{n+1}(A)$ and that both rings are equipped with a $(p, n+1)$-Koszul grading. Since $R$ is finite dimensional, it follows that $R_\bullet \simeq \bigoplus_{i \geq 0} \rad(R)^i/\rad(R)^{i+1}$, and therefore $R_\bullet$ and $\Pi_{n+1}(A)_\bullet$ are graded isomorphic. Hence, transporting the higher preprojective grading on $\Pi_{n+1}(A)$ along this isomorphism, we obtain a higher preprojective cut on $R_\bullet$, and its degree $0$ part is isomorphic to $A$.  
\end{proof}

\section{The \texorpdfstring{$n$}{n}-representation infinite case}
In this section, we discuss the situation where $R_\bullet$ is a Koszul algebra, which is equipped with a higher preprojective grading ${}_\bullet R = \Pi_{n+1}(A)$ such that $A$ is an $n$-representation infinite algebra. We follow the same strategy as for the $n$-representation finite case. However, before computing some $\operatorname{Ext}^\ast(S,S)$ algebra, we need to argue that we can choose the same set of graded simple modules for both gradings. 

\subsection{Graded radicals and Graded Simples}
The purpose of this section is to discuss the graded simples of ${}_\bullet R$ and $R_\bullet$. We first show that ${}_\bullet R$ and $R_\bullet$ have isomorphic semisimple base rings in degree $0$. However, we do not obtain that the graded simples coincide as ungraded modules, so we will need additional assumptions. 

We start with the following trivial remark. 
\begin{remark}
    Let $A$ and $B$ be finite dimensional algebras. If there exist injective algebra morphisms $A \to B$ and $B \to A$, then $A \simeq B$. 
\end{remark}

\begin{prop}
    Let $R_\bullet$ and ${}_\bullet R$ be two \emph{arbitrary}, nonnegative gradings on an algebra $R$, and write $A = R_0$ and $B = {}_0 R$. If both gradings are locally finite dimensional, then $\topp(A) \simeq \topp(B)$. 
\end{prop}

\begin{proof}
    We write $I = \grad(R_\bullet) $ and $J = \grad({}_\bullet R)$.
    By the Wedderburn-Malcev theorem, we can choose semisimple embeddings $S\leq A$ and $S' \leq B$ of the algebras $\topp(A)$ and $\topp(B)$. Then we consider the composition 
    \[ \varphi \colon S \to R \to R/J = \topp(B). \]
    The first map is an embedding and therefore injective. We claim that the composition is also injective. Assume to the contrary that it is not. Then $\ker(\varphi)$ is an ideal in $S$, and since $S$ is semisimple it follows from \Cref{Lem: Gradings on semisimple are trivial} that $\ker(\varphi)$ contains an idempotent $e$. Thus, we have that $e \in J$. We decompose $e$ with respect to the second grading as $e = {}_0 e + {}_+ e$ where ${}_0 e \in \rad(B)$ and ${}_+ e \in {}_+ R$. Since $B$ is finite dimensional, $\rad(B)$ is nilpotent so we find that 
    \[ {}_0 e + {}_+ e = e = e^n = ({}_0 e)^n + y = 0 + y  \]
    for some large enough $n \in \mathbb{N}$ and $y \in {}_+ R$, and hence we see that $e = {}_+ e$. Note that $e^n \in {}_{\geq n} R$ for all $n$. Since $S$ is finite dimensional, we can not have $S \cap {}_{\geq n} R \neq 0$ for all $n$, so there exists some $n$ with $ e = e^n = ({}_+ e)^n = 0$, and the claim follows. 
    Repeating the argument for $S'$, we see that $S$ embeds into $S'$ and $S'$ embeds into $S$, and hence we have $\topp(A) \simeq S \simeq S' \simeq  \topp(B)$. 
\end{proof}

\begin{cor}
    With the above assumptions, we have an isomorphism of algebras 
    \[ R_0 \simeq {}_0 R/\rad({}_0 R). \] 
\end{cor}

However, it does not follow immediately that the graded simples of $R_\bullet$ and ${}_\bullet R$ coincide as ungraded modules. For this, we need the extra assumption that the graded radicals coincide.  

\begin{prop}\label{Prop: Graded simples coincide if graded radicals do}
    Let $R_\bullet$ and ${}_\bullet R$ be nonnegative gradings on a ring $R$. If $\grad(R_\bullet) = \grad({}_\bullet R)$, then the graded simple modules of $R_\bullet$ and ${}_\bullet R$ coincide as ungraded modules.
\end{prop}

\begin{proof}
    Since both gradings are nonnegative, it follows that a graded simple module is concentrated in a single degree, where it is a simple module over the degree $0$ part. That means a sum of representatives of the graded simple modules as ungraded modules is given by 
    \[ \topp(R_0) = R_\bullet / \grad(R_\bullet) = {}_\bullet R / \grad({}_\bullet R) = \topp({}_0 R). \qedhere \]
\end{proof}

\subsection{The graded radicals}
In light of \Cref{Prop: Graded simples coincide if graded radicals do}, we now want to investigate under which assumptions the graded radicals coincide. This is particularly nice when we are working with graded basic algebras, so we assume that $S = R_0$ and $A = {}_0 R$ are basic. 

\begin{prop} \label{Prop: Nilpotent element has radical degree 0}
    Let $R_\bullet$ be an arbitrary, locally finite grading so that $R_0$ is basic. If $r \in R$ is nilpotent, then $r_0 \in \rad(R_0)$.  
\end{prop}

\begin{proof}
    Decompose $r = r_0 + r_+$ with respect to the grading. It suffices to note that if $r$ is nilpotent, then so is $r_0$. To see this, simply note that 
    \[ 0 = r^n = (r_0 + r_+)^n = r_0^n + x \]
    for some $x \in R_+$. Hence we have $r_0^n = 0$. Next, suppose that $r_0 \not \in \rad(R_0)$. Then we choose a Wedderburn-Malcev decomposition $R_0 = S \oplus \rad(R_0)$, and decompose $r_0 = s + y$ for some $s \in S$ and $y \in \rad(R_0)$. Repeating the same argument, we see that $s$ is nilpotent, but since we assumed that $R_0$ is basic it follows that $s= 0$ and hence $r_0 \in \rad(R_0)$.
\end{proof}

\begin{cor}\label{Cor: Nilpotent generators imply equal gradicals}
    Let $R_\bullet$ and ${}_\bullet R$ be Koszul respectively preprojective gradings as before. Assume that $R_0$ and ${}_0 R$ are basic. If $R_1$ is generated over $R_0$ by nilpotent elements, then $\grad(R_\bullet) \subseteq \grad({}_\bullet R)$, and if ${}_1 R$ is generated over $A = {}_0 R$ by nilpotent elements, then $\grad(R_\bullet) \supseteq \grad({}_\bullet R)$.
\end{cor}

\begin{proof}
    Note that $R$ is generated in degrees $0$ and $1$ for both gradings. Furthermore, $\grad(R_\bullet)$ is generated as an ideal by $R_1$, so if it is generated by a set of nilpotent elements, we obtain from \Cref{Prop: Nilpotent element has radical degree 0} that this set lies in $\rad(A) \oplus {}_+ R = \grad({}_\bullet R)$, and hence we have $\grad(R_\bullet) \subseteq \grad({}_\bullet R)$. Similarly, $\grad({}_\bullet R)$ is generated as an ideal by $\rad(A) \oplus {}_1 R$. Since $A$ is finite dimensional, $\rad(A)$ is generated by a set of nilpotent elements, and by assumption this extends to a generating set of nilpotent elements for $\grad({}_\bullet R)$, so \Cref{Prop: Nilpotent element has radical degree 0} shows $\grad(R_\bullet) \supseteq \grad({}_\bullet R)$.
\end{proof}

\subsection{Nilpotency of generators}\label{SSec: Nilpotency of Generators}
In this section, we assume that ${}_\bullet R$ is a higher preprojective grading so that ${}_0 R$ is $n$-representation infinite. We show that the assumption on ${}_1 R$ being generated by nilpotent elements is fulfilled if we assume that the Gabriel quiver of ${}_0 R = A$ is acyclic. The question whether the Gabriel quiver of an $n$-hereditary algebra is necessarily acyclic has been raised in \cite[Question 5.9]{HIO14} and remains, to the authors' knowledge, open. 

We recall that we may interpret ${}_1R$ as 
\[ \Ext^n_A(D(A), A) \simeq \Hom_A(A, \tau_n^- (A)), \]
so we analyse the structure of this space as an $A$-bimodule. This space can be decomposed according to the decomposition of $A = \bigoplus_{i} P_i$ into indecomposable projectives, so we write
\[ \Hom_A(A, \tau_n^-(A)) = \bigoplus_{i,j}\Hom_A(P_i, \tau_n^-(P_j)). \]
Clearly, any element $f \in \Hom_A(P_i, \tau_n^-(P_j))$ for \emph{non-isomorphic} $P_i \not \simeq P_j$ squares to $0$ in the preprojective algebra, since that square is just the composition 
    \[ P_i \xrightarrow[]{\begin{bmatrix} f \\ 0    \end{bmatrix}} \tau_n^-(P_j) \oplus \tau_n^-(P_i) \xrightarrow[]{\begin{bmatrix}
        0 & \tau_n^-(f)   \end{bmatrix}} \tau_n^{-2}(P_j).   \]
Thus, it remains to deal with morphisms $f \in \Hom_A(P_i, \tau_n^-(P_i))$. 

\begin{remark}
By \cite[Theorem 4.24]{HIO14}, there exists an $n$-almost split sequence 
\[ 0 \to P \xrightarrow[]{f_n} C_n \xrightarrow[]{f_{n-1}} \cdots \xrightarrow[]{f_1} C_1 \xrightarrow[]{f_0} \tau_n^-(P) \to 0  \]
in the additive closure $\mathcal{P} = \operatorname{add}\{ \tau_n^{-i}(A) \mid i \geq 0 \}$. 
By \cite[Proposition 2.3 (b)]{HIO14} this additive closure is \emph{directed} in the sense that 
$$\Hom_{\mathcal{P}}(\tau_n^{-j}(A), \tau_n^{-i}(A)) = 0 $$ for $j > i$. In the following, we write $C_{n+1} = P$ and $C_0 = \tau_n^-(P)$ for convenience.  
\end{remark}

Thus, the terms $C_i$ in the $n$-almost split sequence consist of projective summands or iterated inverse translates of projectives. We first show that the appearance of a indecomposable projective summand $Q$ in $C_i$ means that there is a path in the Gabriel quiver of $A$ from the vertex supporting $P$ to the one supporting $Q$.

\begin{lem}
    Let $Q_i$ be an indecomposable projective summand of $C_i$ for some $1 \leq i \leq n$. Then there exist indecomposable projective summands $Q_j$ of $C_j$ for $i < j \leq n+1$ such that every $f_{j-1} $ induces a non-zero morphism $ Q_j \to Q_{j-1}$. Furthermore, $Q_{n+1} = P$.  
\end{lem}

\begin{proof}
    The proof is inductive in nature. We denote by $\pi_i \colon C_i \to Q_i$ the projection. We claim that $\pi_i \circ f_i \neq 0$. Indeed if $\pi_i \circ f_i = 0$, consider the part of the $n$-almost split sequence 
    \[
    \begin{tikzcd}
C_{i+1} \arrow[r, "f_i"] & C_i \arrow[r, "f_{i-1}"] & C_{i-1}
\end{tikzcd}
    \]
    and apply $\Hom_{\mathcal{P}}(-, Q_i)$. We obtain 
    \[
    \begin{tikzcd}
{\Hom_{\mathcal{P}}(C_{i+1}, Q_i)} & {\Hom_{\mathcal{P}}(C_{i}, Q_i)} \arrow[l, " - \circ f_i"']    & {\Hom_{\mathcal{P}}(C_{i-1}, Q_i)} \arrow[l, "- \circ f_i"'] 
\end{tikzcd}
    \]
    Since the original diagram was part of an $n$-almost split sequence, the resulting diagram is exact. By assumption, we have that $\pi_i \in \Hom_{\mathcal{P}}(C_i, Q_i) $ is in $\operatorname{Ker}(- \circ f_i)$. Hence by exactness, we have that $\pi_i \in \operatorname{Im}(- \circ f_{i-1})$. This means $\pi = g \circ f_{i-1}$ for some $g$, which contradicts the fact that $f_{i-1}$ is a radical morphism.
    
    Thus, the composition $\pi_i \circ f_i$ is nonzero. Since $Q_i$ is projective, it follows that there exists a projective summand $Q_{i+1}$ of $C_{i+1}$ such that the restriction $(\pi_i \circ f_{i})|_{Q_{i+1}} $ is nonzero. Repeating the same argument for $Q_{i+1}$, we obtain a summand $Q_{i+2}$ of $C_{i+2}$, until we find a projective summand $Q_{n+1}$ of $C_{n+1} = P$, which has to be $Q_{n+1} = P$ since $P$ was assumed to be indecomposable. 
\end{proof}

The following corollary is immediate.

\begin{cor}\label{Cor: Projective summands give cycles}
    Using the same notation, if $P$ appears as a summand in some $C_i$ for $i < n+1$, then there is a cycle in the Gabriel quiver of $A$. 
\end{cor}

We now apply this to show that any morphism $P \to \tau_n^-(P)$ needs to factor in $\mathcal{P}$, unless $A$ has a cycle in its Gabriel quiver. 

\begin{prop}\label{Pro: Factorisation}
    Let $f \colon P  \to \tau_n^-(P)$ be an arbitrary morphism, and consider as before the $n$-almost split sequence 
    \[ 0 \to P \xrightarrow[]{f_n} C_n \xrightarrow[]{f_{n-1}} \cdots \xrightarrow[]{f_1} C_1 \xrightarrow[]{f_0} \tau_n^-(P) \to 0.  \]
    Then the following hold:
    \begin{enumerate}
        \item The morphism $f$ factors through $f_0$.
        \item The module $C_1 = Q_1 \oplus \tau_n^-(Q_1')$ for projective modules $Q_1$ and $Q_1'$.
        \item The module $P$ is not a direct summand of $Q_1'$.
        \item If $P$ is a direct summand of $Q_1$, then the Gabriel quiver of $A$ has a cycle.
    \end{enumerate}
\end{prop}

\begin{proof}
    \begin{enumerate}
        \item This is immediate because $f_0$ is right almost split. 
        \item By \cite[Lemma 4.25]{HIO14}, the map $f_0$ is a sink map in $\mathcal{P}$. The statement then follows from $f_0$ being right minimal and $\Hom_{\mathcal{P}}(\tau_n^{-j}(A), \tau_n^-(P))=0$ for $j \geq 2$. 
        \item Suppose that $P$ was a summand of $Q_1'$. Then $f_0$ restricted to $\tau_n^-(P)$ gives an irreducible morphism in $\Hom_{\mathcal{P}}(\tau_n^-(P), \tau_n^-(P))$. Applying $\tau_n$ then produces an irreducible morphism $g = \tau_n(f_0|_{\tau_n^-(P)})$ in $\Hom_{\mathcal{P}}(P, P)$. This means that $g$ is irreducible in $\operatorname{proj}(A)$, and so the Gabriel quiver of the finite dimensional algebra $A$ of finite global dimension has a loop, contradicting the no-loop theorem \cite{NoLoops}. 
        \item This follows immediately by applying \Cref{Cor: Projective summands give cycles}.  \qedhere
    \end{enumerate}
\end{proof}

We obtain the claimed nilpotency as a corollary for \emph{acyclic} $n$-representation infinite algebras. 

\begin{cor}\label{Cor: Acyclicity implies nilpotent generators}
    Let $A$ be an acyclic $n$-representation infinite algebra. Then the $A$-bimodule $\Hom_A(A, \tau_n^-(A))$ is generated by a set of elements that are nilpotent in the preprojective algebra. In particular, $\Pi_{n+1}(A)$ is generated as an algebra by $A$ and some nilpotent elements. 
\end{cor}

\begin{proof}
    We generate $\Hom_A(A, \tau_n^-(A))$ by generating $\Hom_A(P_i, \tau_n^-(P_j))$ for every pair $(i,j)$. If $i \neq j$, all elements in $\Hom_A(P_i, \tau_n^-(P_j))$ square to $0$, and so we pick any basis of $\Hom_A(P_i, \tau_n^-(P_j))$ and add it to our generating set. If $i=j$, then every element in $\Hom_A(P_i, \tau_n^-(P_i))$ factors through $Q_1 \oplus \tau_n^-(Q_1)$ where $Q_1$ and $Q_1'$ are projective modules without direct summands isomorphic to $P_i$. Hence, every element in $\Hom_A(P_i, \tau_n^-(P_i))$ is generated by the previously chosen basis elements.  
\end{proof}

\subsection{Koszulity of \texorpdfstring{$n$}{n}-representation infinite algebras}
In this section, we assume that $A = {}_0 R$ is $n$-representation infinite and that the graded simples of ${}_{\bullet}R$ and $R_{\bullet}$ are isomorphic as ungraded modules. As before, we assume that both gradings are locally finite dimensional. Then we aim to show that $A$ is Koszul. We need the following notion of minimality for graded projective resolutions, which can be found in \cite[Section 2]{MM11}. 

\begin{remark}\label{Rem: Minimal graded resolutions}
    Let ${}_\bullet R$ be nonnegatively graded and locally finite dimensional, and let $M$ be a graded module whose support is bounded below. Then one can construct a resolution 
    $$\cdots \to P^i \xrightarrow[]{d_i} P^{i+1} \to \cdots \to P^0$$ 
    of $M$ by graded projective ${}_\bullet R$-modules $P^i$. This resolution is called minimal if $\operatorname{Im}(d_i) \subseteq \grad(P^{i+1})$. Taking the degree $0$ part of this resolutions then yields a minimal projective resolution of the ${}_0 R$-module ${}_0 M$ by projective ${}_0 R$-modules.  
\end{remark}

The following lemma is a consequence of the fact that the gradings we consider are locally finite. The proof is almost verbatim that of \cite[Proposition 3.1.2]{Madsen}. 

\begin{lem}\label{Lem: Ext decomposes preprojectively}
    Let $R_\bullet$ be a locally finite Koszul grading and ${}_\bullet R$ be a higher preprojective grading such that $A = {}_0 R$ is $n$-representation infinite, and write $S = \topp(A)$. If $S \simeq R_0 $ as ungraded $R$-modules, then 
    \[  \Ext^i_R(S, S) = \bigoplus_{j \in \mathbb{Z}} \Ext^i_{ {}_\bullet R }(S, S\langle j \rangle ).  \]
\end{lem}

\begin{proof}
    Since $S$ is concentrated in degree $0$, it follows from the proof of \cite[Proposition 3.1.2]{Madsen} that 
    \[ \Ext_R^i(S,S) \simeq \prod_{j \in \mathbb{Z}} \Ext^i_{{}_\bullet R}(S, S \langle j \rangle) \]
    for each $i \geq 0$. Since $R_\bullet$ is Koszul, we have that $\Ext^0_R(S,S) \simeq D(R_0)$ and $\Ext^1_R(S,S) \simeq D(R_1)$, which both are finite dimensional by assumption. Furthermore, since $R_\bullet$ is Koszul, we have that $\Ext^\ast_R(S,S)$ is generated in cohomological degrees $0$ and $1$, so it follows that $\Ext_R^i(S,S)$ is finite dimensional for all $i \geq 0$. Therefore, each direct product $\prod_{j \in \mathbb{Z}} \Ext^i_{{}_\bullet R}(S, S \langle j \rangle)$ is a finite direct product, and hence we obtain  
    \[ \Ext_R^i(S,S) \simeq \prod_{j \in \mathbb{Z}} \Ext^i_{{}_\bullet R}(S, S \langle j \rangle) \simeq \bigoplus_{j \in \mathbb{Z}} \Ext^i_{{}_\bullet R}(S, S \langle j \rangle). \qedhere \]
\end{proof}

\begin{prop}\label{Prop: Degree 0 is Koszul}
Let $R_\bullet$ be a locally finite Koszul grading and ${}_\bullet R$ be a higher preprojective grading such that $A = {}_0 R$ is $n$-representation infinite. Suppose that the graded simple modules of $R_\bullet$ and ${}_\bullet R$ coincide as ungraded modules. Then $A$ is a Koszul algebra. 
\end{prop}

\begin{proof}
Let $S = \topp(A)$. Hence, by assumption, we have $S \simeq R_0$ as $R$-modules. Since $R_\bullet$ is Koszul, we deduce from \Cref{Prop: Graded algebra Koszul iff Yoneda is generated in 0 and 1} that 
\[ \Ext^{*}_{R}(S,S) = \bigoplus_{i \geq 0}\Ext^{i}_{R}(S,S) \]
is generated in cohomological degree $0$ and $1$. It follows from \Cref{Lem: Ext decomposes preprojectively} that
\[ \Ext^{i}_{R}(S,S) \simeq \bigoplus_{j \in \mathbb{Z}}\Ext^{i}_{ {}_{\bullet}R}(S,S\langle j \rangle). \]  
It is straightforward to check that 
\[ ({}_{0}R)^! = \bigoplus_{i \in \mathbb{Z}}\Ext^{i}_{ {}_{\bullet}R}(S,S) \]
is a subalgebra of $R^! = \Ext^{*}_{R}(S,S)$, and we claim that it is also generated in cohomological degree $0$ and $1$. Assume to the contrary that it has a generator $x$ in cohomological degree $> 1$. Since $x$ is also an element in $R^!$, it can be expressed as a sum of products of generators of $R^!$ that lie in degree $0$ and $1$.
However, as ${}_{\bullet}R$ is positively graded,  we have that $R^!$ has no generators lying in some
\[ \Ext^{i}_{{}_{\bullet}R}(S,S\langle j \rangle) \]
for $j < 0$, which would be necessary if ${}_{0}R^!$ were to have a generator in cohomological degree $> 1$, and hence there are no such generators $x$. 

Now, we consider a minimal graded projective resolution $P^{\bullet}$ of $S$ over ${}_{\bullet}R$ as in \Cref{Rem: Minimal graded resolutions}. Then its degree $0$ part is a minimal projective resolution of $S$ over $A$.
Observe that $\grad({}_\bullet R)$ annihilates $S$, and that thus we can write 
\begin{align*}
\Ext^{i}_{{}_{\bullet}R}(S,S) 
& \simeq \Hom_{{}_{\bullet}R}(P^{i},S)\\
& \simeq \Hom_{A}({}_{0}P^{i},S)\\
& \simeq \Ext^{i}_{A}(S,S),
\end{align*}
from which we deduce that $\Ext^{*}_{A}(S,S)$ is generated in degree $0$ and $1$, and hence $A$ is Koszul by \Cref{Cor: FD algebera Koszul iff Yoneda is gen in 0 and 1}. 
\end{proof}

\subsection{Preprojective gradings and cuts}
We are now ready to state our main theorem for the $n$-representation infinite case. In light of \Cref{SSec: Nilpotency of Generators}, we fix the following terminology. 

\begin{defin}
    A higher preprojective grading ${}_\bullet R$ on some algebra $R$ is called \emph{acyclic} if the finite-dimensional algebra ${}_0 R$ has acyclic Gabriel quiver.
\end{defin}

\begin{thm}\label{Thm: Main theorem for basic nRI}
    Let $R = R_\bullet$ be a graded basic Koszul algebra, so that $R_1$ is generated by nilpotent elements. If there exists an acyclic higher preprojective grading ${}_\bullet R$ on $R$, then there exists an automorphism that maps this higher preprojective grading to a higher preprojective cut. In particular, all acyclic higher preprojective gradings are, up to isomorphism, higher preprojective cuts.  
\end{thm}

\begin{proof}
    Combining \Cref{Cor: Acyclicity implies nilpotent generators} and \Cref{Cor: Nilpotent generators imply equal gradicals}, we obtain that the graded simples for the Koszul grading and the higher preprojective grading coincide. Therefore, we can apply \Cref{Prop: Degree 0 is Koszul} to see that $A = {}_0 R$ is Koszul. Then, by \Cref{Prop: GI grading defines a cut}, there exists a Koszul grading on $\Pi_{n+1}(A) \simeq R$ so that the higher preprojective grading on $\Pi_{n+1}(A)$ is a cut. Finally, by \Cref{Cor: Koszul grading is unique}, there exists an automorphism taking the new Koszul grading on $\Pi_{n+1}(A) \simeq R$ to the given grading $R_\bullet$. Hence, this automorphism maps ${}_\bullet R$ to a cut.  
\end{proof}

We also have the following non-basic version, assuming that we can work over a fixed semisimple base ring. 

\begin{thm}\label{Thm: Main theorem for nonbasic nRI}
    Let $R = R_\bullet$ be a Koszul algebra, so that $R_1$ is generated by nilpotent elements. If there exists an acyclic higher preprojective grading ${}_\bullet R$ on $R$ such that $R_0 \subseteq {}_0 R$, then there exists a higher preprojective cut realising $R_\bullet = \Pi_{n+1}(A)$ so that $A \simeq {}_0 R$. 
\end{thm}

\begin{proof}
    By assumption, we can reduce $R_\bullet$ to a graded basic algebra $R_\bullet^b$ by some homogeneous idempotent $e \in R_0 \subseteq {}_0 R$. Since $e$ is preprojectively homogeneous, it follows that $eRe$ inherits a higher preprojective grading. By \Cref{Thm: Main theorem for basic nRI}, we obtain a higher preprojective cut, and by \Cref{Prop: HPC pulls back along Morita} this gives the higher preprojective cut on $R_\bullet$.
\end{proof}

We conclude with an application showing that $n$-APR tilting preserves being Koszul. For the definition of $n$-APR tilting, we refer the reader to \cite{Iyama-Oppermann}. 

\begin{cor}
    Let $A$ be basic and $n$-representation infinite, and let $B$ be an $n$-APR tilt of $A$. Then $A$ is Koszul if and only if $B$ is Koszul.
\end{cor} 

\begin{proof}
    If $A$ is Koszul, then so is $\Pi_{n+1}(A)$ by \Cref{Thm: GI construction}, and the higher preprojective grading is a cut for this grading. Hence, the graded simple modules for the two gradings coincide as ungraded modules. Transferring the grading to $\Pi_{n+1}(B)$, we get a Koszul grading on $\Pi_{n+1}(B)$ satisfying the assumptions of \Cref{Prop: Degree 0 is Koszul}, showing that $B$ is Koszul. 
\end{proof}

\section*{Acknowledgements}
We thank Gustavo Jasso for pointing out an error in a previous version of this article. We thank Steffen Oppermann for helpful discussions, and for suggesting the proof of a Lemma which was used in a previous version of this article. We also thank Martin Herschend for many helpful discussions. The second author is grateful to have been supported by Norwegian Research Council project 301375, ``Applications of reduction techniques and computations in representation theory''.

\bibliographystyle{alpha}

\bibliography{big_awful_bib.bib}
\end{document}